\documentclass[letterpaper, 10 pt, conference]{ieeeconf}%

\IEEEoverridecommandlockouts%

\usepackage{graphics}%
\usepackage{epsfig}%
\usepackage{graphicx}%
\usepackage{times}%
\usepackage{amsmath}%
\usepackage{amssymb}%
\usepackage{algorithm,algorithmic}
\usepackage{subcaption}
\usepackage{url}

\title{\LARGE \bf
Major-Minor LQ Mean Field Games with Erroneous Initial Information: Distributed Error Estimation and Strategy Modification 
}

\author{Yuxin Jin, Haotian Wang, Wang Yao\textsuperscript{*} and Xiao Zhang\textsuperscript{*}%
\thanks{* This work was supported by the National Natural Science Foundation of China (NSFC Grant No.12441101).(Corresponding author: Wang Yao and Xiao Zhang).}%
\thanks{This work has been submitted to the IEEE for possible publication. Copyright may be transferred without notice, after which this version may no longer be accessible.}
\thanks{Y. Jin thanks Prof. Markus Fischer and Prof. Alekos Cecchin for helpful discussions. Y. Jin also gratefully acknowledges the financial support from the China Scholarship Council (CSC).}
\thanks{Y. Jin is with ShenYuan Honors College and School of Mathematical Sciences, Beihang University, Beijing 100191, China; Key Laboratory of Mathematics, Informatics and Behavioral Semantics, Ministry of Education, Beihang University, Beijing 100191, China (e-mail: yxjin@buaa.edu.cn).}
\thanks{H. Wang is with School of Mathematical Sciences, Beihang University, Beijing 100191, China (e-mail: williamelwht@buaa.edu.cn).}%
\thanks{W. Yao is with School of Artificial Intelligence and LMIB, Beihang University, Beijing 100191, China; Hangzhou International Innovation Institute of Beihang University, Hangzhou 311115, China (e-mail: yaowang@buaa.edu.cn).}%
\thanks{X. Zhang is with the School of Mathematical Sciences, Beihang University, Beijing 100191, China; Key Laboratory of Mathematics, Informatics and Behavioral Semantics, Ministry of Education, Beihang University, Beijing 100191, China (e-mail: xiao.zh@buaa.edu.cn).}%
}

\begin{document}
\maketitle
\thispagestyle{empty}
\pagestyle{empty}
\newtheorem{remark}{Remark}[section]
\newtheorem{lemma}{Lemma}[section]
\newtheorem{theorem}{Theorem}[section]
\newtheorem{problem}{Problem}[section]
\newtheorem{assumption}{Assumption}[section]
\newtheorem{proposition}{Proposition}[section]
\newcommand{\proofsketchqed}{\hfill\vrule height 1.2ex width 0.9ex depth 0pt}

\newenvironment{proofsketch}
{\par\noindent\textit{Proof sketch: }\ignorespaces}
{\proofsketchqed\par}
\begin{abstract}
This paper studies major-minor linear-quadratic mean field games (MMLQMFGs) with erroneous initial information under a constrained observation structure. Each minor agent observes only its own state and the major agent’s state, while the major agent observes its own state and the states of a subset of minor agents; neither side observes the mean field state directly. We show that the initial-information errors propagate linearly through the game dynamics and lead to explicit deviations in the major state, the actual mean field, and the agents’ internally updated mean field states. Based on this structure, we formulate distributed error identification as a parameter-estimation problem from discrete-time local observations and construct maximum-likelihood estimators for unknown initial errors. We then propose an estimate-based strategy modification at an intermediate time by reconstructing the current mean field from the estimated errors and switching to the corresponding control law. We also characterize the resulting estimation errors and show that, in the present symmetric setting, the major agent's estimation precision depends on the number of observed minor agents but not on their identities. Numerical results illustrate the proposed method.
\end{abstract}

\section{Introduction}

Mean field game (MFG) theory provides a tractable framework for studying strategic interactions in very large populations of weakly coupled agents. Since the foundational works of Lasry and Lions and of Huang, Caines, and Malham\'e, it has become a tool for decentralized decision making in large-scale stochastic systems~\cite{lasry2007}-\cite{carmona2018}. Among its extensions, major-minor mean field games (MM-MFGs) are of particular interest. In an MM-MFG, one major agent has an asymptotically non-vanishing influence on the population, whereas each minor agent is individually negligible; consequently, the limiting mean field is typically stochastic through its dependence on the major state~\cite{huang2010major}-\cite{huang2021major}.

In many applications, however, the information assumptions used in standard MFG models are idealized. In particular, the mean field is often not directly observable, and agents may have access only to local or partial observations. This has led to a growing literature on partially observed and information-constrained MFGs, including settings with partial observations and related imperfect-information models~\cite{sen2016major}-\cite{shmaya2025}.

This paper focuses on a different but related issue: agents may start with erroneous initial information on the mean field and be unable to observe the true mean field later in real time. Such errors may come from inaccurate initialization, communication delays, or stale broadcasts~\cite{zhang2020}-\cite{yates2021}. In the homogeneous-population setting, erroneous initial information under partial and discrete observations has been studied in~\cite{jin2025iet}. The corresponding major-minor case, however, is still open.

In our model, each minor agent observes only its own state and the major agent’s state, while the major agent observes its own state and the states of a subset of minor agents. Neither side observes the mean field state directly. Under correct initial information, agents can recursively update the mean field state from the observed major state. Under erroneous initial information, this internal update becomes biased, which further affects the control actions and the system evolution.

In this paper, we study a major-minor linear-quadratic mean field game under erroneous initial information and a constrained observation structure. The main contributions are as follows.

\begin{enumerate}
    \item We derive an explicit linear propagation law for the effects of the initial information errors on the major state, the actual mean field, and the agents’ internally updated mean-field states.
    
    \item Based on this structure, we reformulate distributed error identification as a parameter estimation problem from discrete-time local observations. Each minor agent estimates its own private error together with the common error components, while the major agent estimates the common components.
    
    \item We propose an estimate-based one-shot strategy modification at an intermediate time. Each agent reconstructs the current mean field from the estimated errors and switches to the corresponding control law.
    
    \item We further analyze the resulting estimators. We characterize the estimation errors for both major and minor agents, and show that, for the major agent, the estimation precision depends on the number of observed minor agents rather than their identities.
\end{enumerate}

The rest of the paper is organized as follows. Section~II reviews the benchmark MMLQMFG under correct information. Section~III studies the effect of erroneous initial information and derives the linear error-propagation relations. Section~IV formulates the corresponding likelihood-based parameter-estimation problem. Section~V develops the estimate-based mean-field reconstruction and strategy modification scheme. Section~VI analyzes the estimation errors. Section~VII presents a numerical example.

\section{Major-Minor Linear Quadratic Mean Field Game under Correct Information}
In this section, we introduce the MMLQMFG model under correct initial information. We consider a game with $N$ minor agents, $\mathcal{A}_i,1\leq i\leq N$, $N\rightarrow \infty$ and one major agent $\mathcal{A}_0$. We derive the open-loop equilibrium and give the open-loop equilibrium system together with the induced feedback control law for agents \cite{huang2010major,ma2020major,huang2021major}.
\subsection{Information Structure}
Let $(\Omega,\mathcal F,\mathbb{P})$ support a $d_0$-dimensional Brownian motion $W^0$ and, for each minor agent $\mathcal{A}_i$, a $d$-dimensional Brownian motion $W^i$, where $W^0$ and $W^i, 1\leq i\leq N$ are independent from each other. Let
\[\mathbb{F}^0=(\mathcal{F}_t^0)_{t\in[0,T]},\mathcal{F}_t^0:=\sigma(W_s^0:0\le s\le t).\]
Each minor agent $\mathcal{A}_i$ observes its own state and the major $\mathcal A_0$'s state, and the major $\mathcal{A}_0$ observes its own state and a subset of minors' states $\{X^k\}_{k\in\mathcal{I}_{obs}}$, where $X^k_t$ represents $\mathcal{A}_k$'s state at time $t$, and $\mathcal{I}_{obs}$ is a fixed index set. At the initial moment, agents get the initial average state of minor agents. 

\subsection{Dynamics and Cost Functionals}
Let $(X^i_t)_{0\leq t\leq T}$ be the state of $\mathcal{A}_i$, and $x_0^i$ represents the initial state of $\mathcal{A}_i$, $0\leq i\leq N$. The dynamics of minor agent $\mathcal{A}_i$ is given by
\begin{equation}
\label{eq:Xi}
\begin{split}
dX_t^i = \big(AX_t^i + G X_t^0 + F z_t + B u_t^i\big)dt + DdW_t^i, X_0^i=x_0^i,
\end{split}
\end{equation}
and the dynamics of the major agent $\mathcal{A}_0$ is
\begin{equation}
\label{eq:X0}
\begin{split}
dX_t^0 = \big(A_0X_t^0+F_0 z_t + B_0 u_t^0\big)dt + D_0dW_t^0, X_0^0=x_0^0,
\end{split}
\end{equation}
where $X_t^0\in\mathbb{R}^{d_0}, X_t^i\in\mathbb{R}^d$, $A, B, G, F, D, A_0, B_0$, $F_0, D_0$ are matrices of suitable sizes, $z_t:=1/N\Sigma_{j=1}^N X^j_t$ and $\bar{u}_t:=1/N\Sigma_{j=1}^N u^j_t$ are the mean field state and mean field control. Define $\mathcal{I}_{\mathrm{obs}}$ as the set of minors' index whose state is observable by $\mathcal{A}_0$. The admissible control $u^i_t$ is in $\mathcal{U}^i:=L^2_{\mathcal{G}^i_t}(0,T;\mathbb{R}^m)$, where
\[\mathcal{G}_t^i:=\sigma(z_0, X_s^i,X_s^0:0\leq s\leq t), \]
and $u^0_t$ is in $\mathcal{U}^0:=L^2_{\mathcal{G}^0_t}(0,T;\mathbb{R}^{m_0})$, where
\[\mathcal{G}^0_t:=\sigma(z_0, X_s^0,\{X_s^k\}_{k\in\mathcal{I}_{\mathrm{obs}}}:0\le s\le t).\]

The cost functional of minor agent $\mathcal{A}_i$ is 
\begin{equation}
\label{eq:Ji}
\begin{split}
&J(u^i; z)=\frac{1}{2}\mathbb{E}[\int_{0}^{T}[\|u^i_t\|_R^2+\|X^i_t-(\Gamma z_t+\Lambda X_t^0+\eta)\|_Q^2]\\
&dt+\|X^i_T-(\bar{\Gamma}z_T+\bar{\Lambda}X_T^0+\bar{\eta})\|_{\bar{Q}}^2],
\end{split}
\end{equation}
and the cost functional of major agent $\mathcal{A}_0$ is
\begin{equation}
\label{eq:J0}
\begin{split}
&J_0(u^0;\bar{u})=\frac{1}{2}\mathbb{E}[\int_{0}^{T}[\|u^0_t\|_{R_0}^2+\|X^0_t-(\Gamma_0 z_t+\eta_0)\|_{Q_0}^2]dt\\
&+\|X^0_T-(\bar{\Gamma}_0z_T+\bar{\eta}_0)\|_{\bar{Q}_0}^2],
\end{split}
\end{equation}
where we define $\|X\|_Q^2=X^\top QX$. $Q, R, \bar{Q}, \Gamma, \bar{\Gamma}, Q_0, R_0$, $\bar{Q}_0, \Gamma_0, \bar{\Gamma}_0$ are square matrices, and $R,R_0$ are positive definite. $\eta,\bar{\eta}, s, \bar{s}$ are constant vectors. $Q, \bar{Q}, Q_0, \bar{Q}_0\succeq 0$.
\subsection{Optimal Control and Equilibrium}
As $N\rightarrow \infty$, for given continuous $\mathcal{F}^0_t$-adapted process $z$ and $X^0$, according to the stochastic maximum principle \cite{yong1999,oksendal2003}, the optimal control of a minor agent $\mathcal{A}_i$ is $u^i_t=-R^{-1}B^\top p_t^i$, where $p^i$ is the solution to the following adjoint equation:
\begin{equation}\label{eq:pi}
\left\{
\begin{split}
&dp_t^i=-(A^\top p_t^i+Q(X_t^i-\Gamma z_t-\Lambda X^0_t-\eta))dt+Z^i_tdW^i_t\\
&+Z_t^{i,0}dW_t^0, p_T^i=\bar{Q}(X_T^i-\bar{\Gamma} z_T-\bar{\Lambda} X_T^0-\bar{\eta}).
\end{split}
\right.
\end{equation}
Select $X$ as the state of a representative agent $\mathcal{A}$, and $p$ is the corresponding adjoint process, $\bar{p}_t=\mathbb{E}[p_t|\mathcal{F}_t^0]$, then according to the definition of $z$, when minor agents take their optimal control,
\[z_t=\mathbb{E}[X_t|\mathcal{F}^0_t], \bar{u}_t=-R^{-1}B^\top \bar{p}_t, \]
and combining with \eqref{eq:Xi}, we have
\begin{equation}\label{eq:z}
dz_t=\big((A+F)z_t+GX_t^0+B\bar{u}_t\big)dt, z_0=z_0.
\end{equation}
For given $\bar{u}$, the major agent $\mathcal{A}_0$ optimizes $J_0$ subject to:
\begin{equation}\label{eq:X0_z}
\left\{
\begin{split}
&d\binom{X_t^0}{z_t}=\begin{pmatrix}
A_0 & F_0 \\
G & A+F
\end{pmatrix}
\binom{X_t^0}{z_t}+\begin{pmatrix}
B_0 & 0\\
0 & B
\end{pmatrix}\binom{u^0_t}{\bar{u}_t}dt\\
&+\binom{D_0}{0}dW_t^0,\binom{X_0^0}{z_0}=\binom{x_0^0}{z_0}.
\end{split}
\right.
\end{equation}
Let $\zeta_t:=((X_t^0)^\top,z_t^\top)^\top$. Then according to the stochastic maximum principle \cite{yong1999,oksendal2003}, the optimal control of the major agent $\mathcal{A}_0$ is $u^0_t=-R_0^{-1}(B_0^\top,0)p_t^0$, where $p^0$ is the solution to the following adjoint equation:
\begin{equation}\label{eq:p0}
\left\{
\begin{split}
&dp_t^0=-(A_1^\top p_t^0+\Pi_0^\top Q_0(\Pi_0\zeta_t-\eta_0))dt+Z^0_tdW^0_t,\\
&p_T^0=\bar{\Pi}_0^\top \bar{Q}_0(\bar{\Pi}_0\zeta_T-\bar{\eta}_0),
\end{split}
\right.
\end{equation}
where
\[A_1=\begin{pmatrix}
A_0 & F_0\\
G & A+F
\end{pmatrix},\Pi_0^\top=\begin{pmatrix}
I_{d_0}\\
-\Gamma_0^\top
\end{pmatrix},\bar{\Pi}_0^\top=\begin{pmatrix}
I_{d_0}\\
-\bar{\Gamma}_0^\top
\end{pmatrix}.\]
When all agents take their optimal control, we have the following open-loop equilibrium equations:
\begin{equation}\label{eq:equilibrium}
\left\{
\begin{split}
&dz_t=\big((A+F)z_t+GX_t^0-BR^{-1}B^\top\bar{p}_t\big)dt,\\
&dX_t^0=\big(A_0X_t^0+F_0z_t-B_0R_0^{-1}(B_0^\top,0)p_t^0\big)dt+D_0dW_t^0,\\
&d\bar{p}_t=-(A^\top \bar{p}_t+Q((I_d-\Gamma) z_t-\Lambda X^0_t-\eta))dt+Z_tdW^0_t,\\
&dp_t^0=-(A_1^\top p_t^0+\Pi_0^\top Q_0(\Pi_0\zeta_t-\eta_0))dt+Z^0_tdW^0_t,\\
&z_0=z_0,X_0^0=x_0^0,\bar{p}_T=\bar{Q}((I_d-\bar{\Gamma})z_T-\bar{\Lambda} X_T^0-\bar{\eta}),\\
&p_T^0=\bar{\Pi}_0^\top \bar{Q}_0(\bar{\Pi}_0\zeta_T-\bar{\eta}_0).
\end{split}
\right.
\end{equation}
Setting $\mathcal{Y}_t:=({p_t^0}^\top,\bar{p}_t^\top)^\top$, we have $\mathcal{Y}_t=P_t\zeta_t+\mathcal{G}_t$, where
\begin{equation}\label{eq:P}
\left\{
\begin{split}
&-dP_t=(\mathbb{A}^\top P_t+P_tA_1-P_t\mathbb{M}P_t+\mathcal{Q})dt, P_T=\bar{\mathcal{Q}},\\
&d\mathcal{G}_t=-\big((\mathbb{A}^\top-P_t\mathbb{M})\mathcal{G}_t-\nu)dt,\mathcal{G}_T=\bar{\nu},
\end{split}
\right.
\end{equation}
and
\begin{equation}
\begin{split}
&\mathbb{A}=\begin{pmatrix}
A_0 & F_0 &0\\
G & A+F & 0\\
0 & 0 & A\end{pmatrix}, \mathcal{Q}=\begin{pmatrix}
Q_0 & -Q_0\Gamma_0\\
-\Gamma_0^\top Q_0 & \Gamma_0^\top Q_0\Gamma_0\\
-Q\Lambda & -Q(\Gamma-I_d)\end{pmatrix},\\
&\mathbb{M}=\begin{pmatrix}
B_0R_0^{-1}B_0^\top &0 &0\\
0 &0 &BR^{-1}B^\top\end{pmatrix}, \nu=\begin{pmatrix}
\Pi_0^\top Q_0\eta_0\\
Q\eta\end{pmatrix},\\
&\bar{\mathcal{Q}}=\begin{pmatrix}
\bar{Q}_0 & -\bar{Q}_0\bar{\Gamma}_0\\
-\bar \Gamma_0^\top \bar Q_0 & \bar \Gamma_0^\top \bar Q_0\bar \Gamma_0\\
-\bar Q\bar \Lambda & -\bar Q(\bar \Gamma-I_d)\end{pmatrix}, \bar \nu=\begin{pmatrix}
\bar \Pi_0^\top \bar Q_0\bar \eta_0\\
\bar Q\bar \eta\end{pmatrix}.
\end{split}
\end{equation}
Perform matrix partitioning on $P_t$ and $\mathcal{G}_t$:
\[P_t=\begin{pmatrix}
P^{11}_t & P^{12}_t\\
P^{21}_t & P^{22}_t\\
P^{31}_t & P^{32}_t\end{pmatrix}, \mathcal{G}_t=\begin{pmatrix}
\mathcal{G}^{1}_t\\
\mathcal{G}^{2}_t\\
\mathcal{G}^{3}_t\\\end{pmatrix},\]
where $P^{11}_t$ is $d_0\times d_0$ matrix function, $P^{21}, P^{31}$ are $d\times d_0$ matrix functions, we have
\begin{equation}\label{eq:feedback_z_x0}
\left\{
\begin{split}
&u_t^0=-R_0^{-1}B_0^\top(P^{11}_tX_t^0+P^{12}_t z_t+\mathcal{G}^1_t),\\
&\bar{u}_t=-R^{-1}B^\top(P^{31}_tX_t^0+P^{32}_t z_t+\mathcal{G}^3_t).
\end{split}
\right.
\end{equation}
According to \eqref{eq:pi}, we have 
\begin{equation}\label{eq:Px}
\left\{
\begin{split}
&u^i_t=-R^{-1}B^\top(P^x_tX_t^i+(P_t^{32}-P_t^x)z_t+P_t^{31}X_t^0+\mathcal{G}_t^3),\\
&-dP_t^x=(A^\top P_t^x+P_t^xA-P_t^xBR^{-1}B^\top P_t^x+Q)dt,\\
&P_T^x=\bar{Q}.
\end{split}
\right.
\end{equation}
We assume Riccati equations in \eqref{eq:P} and \eqref{eq:Px} have a unique solution $P, P^x$ on $t\in[0,T]$. 
\subsection{Agents' Behavior}
In our setting, at time $t$, agents can get the state of $\mathcal{A}_0$ denoted by $X_t^0$, but can not observe the mean field state $z_t$ directly. Applying the optimal control law \eqref{eq:feedback_z_x0}, we can rewrite \eqref{eq:z} as the following closed-loop form:
\begin{equation}\label{eq:close_z}
\left\{
\begin{split}
&dz_t=\big((A+F-BR^{-1}B^\top P_t^{32})z_t+(G-BR^{-1}B^\top \\
&P_t^{31})X_t^0-BR^{-1}B^\top \mathcal{G}_t^3\big)dt, z_0=z_0.
\end{split}
\right.
\end{equation}

Then by knowing $z_0$, each agent can update and calculate $z_t$ by observed $X_t^0$ and calculated $z_{t-}$. 
At time $t=0$, a minor agent $\mathcal{A}_i$ computes $P,\mathcal{G}, P^x$, the major agent $\mathcal{A}_0$ computes $P,\mathcal{G}$. At time $t>0$, $\mathcal{A}_i$ gets $X_t^i$, $X_t^0$, updates $z_t$, $\mathcal{A}_0$ gets $X_t^0$, updates $z_t$, and all agents apply their optimal control laws.

\section{Major-Minor Linear Quadratic Mean Field Game under Erroneous Initial Information}
In this section, we introduce the MMLQMFG model under erroneous initial information. We consider the initial information error on the mean field state $z_0$:
\begin{equation}\label{eq:error}
z_0^i=z_0+e^i, z_0^0=z_0+e^0, i=1,...,N,
\end{equation} 
where $z_0^i$ is the initial mean field state observed by a minor agent $\mathcal{A}_i$ with deviation $e^i$, and $z_0^0$ is the initial mean field state observed by the major agent with deviation $e^0$, $e^i, e^0\in \mathbb{R}^d$. Define $\bar{e}=\lim_{N\rightarrow\infty}\frac{1}{N}\sum_{i=1}^N e^i$.

We analyze how the initial errors influence the agents' control law and the evolution of the system.
\subsection{Basic Settings}
We keep the information settings in the last section. Since the initial observations on $z_0$ are erroneous now, the admissible control sets are changed. The admissible control set of a minor agent $\mathcal{A}_i$ is $\mathcal{U}^i:=L^2_{\mathcal{G}_t^i}(0,T;\mathbb{R}^m)$, where
\[\mathcal{G}_t^i:=\sigma(z_0^i,X_s^i,X_s^0:0\leq s\leq t),\]
and that of $\mathcal{A}_0$ is $\mathcal{U}^0:=L^2_{\mathcal{G}^0_t}(0,T;\mathbb{R}^{m_0})$, where
\[\mathcal{G}^0_t:=\sigma(z_0^0, X_s^0,\{X_s^k\}_{k\in\mathcal{I}_{\mathrm{obs}}}:0\le s\le t).\]
The mean field state $z$ in the cost functional of a minor agent $\mathcal{A}_i$ is now replaced by $z^i$, where $z^i$ represents the mean field state updated by $\mathcal{A}_i$, and $z$ in the cost functional of $\mathcal{A}_0$ is now replaced by $z^0$, where $z^0$ represents the mean field state updated by $\mathcal{A}_0$.
\subsection{Agents' Behavior}
The initial error will not influence $P, \mathcal{G}, P^x$, and these processes can be calculated according to system parameters. In $\mathcal{A}_i$'s opinion, $z^i$ is the correct mean field state, so the feedback control law of a minor agent $\mathcal{A}_i$ now is 
\begin{equation}\label{eq:ui_error}
u_t^i=-R^{-1}B^\top(P_t^xX_t^i+(P_t^{32}-P_t^x)z_t^i+P_t^{31}X_t^0+\mathcal{G}_t^3),
\end{equation}
and the feedback control law of $\mathcal{A}_0$ is
\begin{equation}\label{eq:u0_error}
u_t^0=-R_0^{-1}B_0^\top(P_t^{11}X_t^0+P_t^{12}z_t^0+\mathcal{G}_t^1).
\end{equation}
The update law of $z^j, j=0,...,N$ becomes
\begin{equation}\label{eq:close_zj}
\left\{
\begin{split}
&dz^j_t=\big((A+F-BR^{-1}B^\top P_t^{32})z^j_t+(G-BR^{-1}B^\top \\
&P_t^{31})X_t^0-BR^{-1}B^\top \mathcal{G}_t^3\big)dt, z^j_0=z_0+e^j.
\end{split}
\right.
\end{equation}
Then the behavior of agents can be described as follows:
 
At time $t=0$, a minor agent $\mathcal{A}_i$ computes $P,\mathcal{G}, P^x$, the major agent $\mathcal{A}_0$ computes $P,\mathcal{G}$. At time $t>0$, $\mathcal{A}_i$ gets $X_t^i$, $X_t^0$, updates $z_t^i$, $\mathcal{A}_0$ gets $X_t^0$, updates $z_t^0$, and all agents apply their optimal control laws \eqref{eq:ui_error} and \eqref{eq:u0_error}.
\subsection{Initial Error Effect}
Denote the mean field state under correct information by $z^c$, the state of major agent under correct information by $X^{0,c}$, the state of a minor agent $\mathcal{A}_i$ under correct information by $X^{i,c}$. Let $M:=BR^{-1}B^\top, M_0:=B_0R_0^{-1}B_0^\top$.

The initial errors first influence the updated mean field state, then agents' control and states, finally the actual mean field state. According to the definition of $z$, we have
\begin{equation}\label{eq:equilibrium_error}
\left\{
\begin{split}
&dX_t^0=\big((A_0-M_0 P^{11}_t)X_t^0-M_0 P_t^{12}z^0_t+F_0z_t-M_0 \\
&\mathcal{G}_t^1\big)dt+D_0dW_t^0, X_0^0=x_0^0,\\
&dz_t^0=\big((A+F-M P_t^{32})z^0_t+(G-MP_t^{31})X_t^0-M \\
&\mathcal{G}_t^3\big)dt,z^0_0=z_0+e^0,\\
&d\bar{z}_t=\big((A+F-M P_t^{32})\bar{z}_t+(G-M P_t^{31})X_t^0-M \\
&\mathcal{G}_t^3\big)dt, \bar{z}_0=z_0+\bar{e},\\
&dz_t=\big((A+F-M P_t^x)z_t+(G-M P_t^{31})X_t^0-M\\
&(P_t^{32}-P_t^x)\bar{z}_t-M\mathcal{G}_t^3\big)dt, z_0=z_0.
\end{split}
\right.
\end{equation}
Set $\delta X_t^0= X_t^0-X_t^{0,c}$, $\delta z_t^0= z_t^0-z_t^{c}$, $\delta \bar{z}_t= \bar{z}_t-z_t^{c}$, $\delta z_t= z_t-z_t^{c}$ and $\mathcal{X}_t:=((\delta X_t^0)^\top, (\delta z_t^0)^\top, (\delta \bar{z}_t)^\top, (\delta z_t)^\top)^\top$. Then compare \eqref{eq:equilibrium_error} with \eqref{eq:close_z} and \eqref{eq:equilibrium}, we have
\begin{equation}\label{eq:deviate_x0_z}
d\mathcal{X}_t=\mathbb{A}^0_t\mathcal{X}_tdt, \mathcal{X}_0=(0^\top, (e^0)^\top, (\bar{e})^\top,0^\top)^\top,
\end{equation}
where we set $A^{22}_t:=A+F-MP^{32}_t$, $A^{44}_t:=A+F-MP_t^x$, and
\[\mathbb{A}_t^0=\begin{pmatrix}
A_0-M_0P^{11}_t & -M_0P^{12}_t &0 & F_0\\
G-MP^{31}_t & A^{22}_t &0 &0\\
G-MP^{31}_t & 0 & A^{22}_t &0\\
G-MP^{31}_t &0 & -M(P_t^{32}-P_t^x) & A^{44}_t
\end{pmatrix}\]
Solving \eqref{eq:deviate_x0_z}, we have
\begin{equation}\label{eq:Phi}
\left\{
\begin{split}
&\delta X_t^0=\Phi_t^{12}e^0+\Phi_t^{13}\bar{e},\delta z_t^0=\Phi_t^{22}e^0+\Phi_t^{23}\bar{e},\\
&\delta \bar{z}_t=\Phi_t^{32}e^0+\Phi_t^{33}\bar{e}, \delta z_t=\Phi_t^{42}e^0+\Phi_t^{43}\bar{e}.\\
\end{split}
\right.
\end{equation}
Where
\[\Phi_t=\begin{pmatrix}
\Phi_t^{11} &\Phi_t^{12} &\Phi_t^{13} &\Phi_t^{14}\\
\Phi_t^{21} &\Phi_t^{22} &\Phi_t^{23} &\Phi_t^{24}\\
\Phi_t^{31} &\Phi_t^{32} &\Phi_t^{33} &\Phi_t^{34}\\
\Phi_t^{41} &\Phi_t^{42} &\Phi_t^{43} &\Phi_t^{44}\\
\end{pmatrix},\]
and
\begin{equation}
d\Phi_t=\mathbb{A}_t^0\Phi_t dt,\Phi_0=I_{3d+d_0}.
\end{equation}
Denote the deviation of $z^i_t$ from $z^c_t$ by $\delta z^i_t$. Then we have
\begin{equation}\label{eq:deviate_zi}
d\delta z^i_t=(A+F-MP^{32}_t)\delta z^i_t+(G-MP^{31}_t)\delta X_t^0 dt, \delta z_0^i=e^i.
\end{equation}  
Applying the method of variation of parameters (see, e.g., \cite{yong1999,oksendal2003}) to solve the above equation, we have
\[\delta z^i_t=\Phi_t^x(e^i+\mathcal{M}_t^0e^0+\bar{\mathcal{M}}_t\bar{e}),\]
where
\begin{equation}\label{eq:Phi_x}
\left\{
\begin{split}
&d\Phi^x_t=(A+F-MP^{32}_t)\Phi_t^xdt,\Phi_0^x=I_{d},\\
&\mathcal{M}_t^0=\int_0^t(\Phi_s^x)^{-1}(G-MP^{31}_s)\Phi^{12}_sds,\\ &\bar{\mathcal{M}}_t=\int_0^t(\Phi_s^x)^{-1}(G-MP^{31}_s)\Phi^{13}_sds.
\end{split}
\right.
\end{equation}
Summarizing the above conclusions, we have the following theorem
\begin{theorem}\label{thm:deviate}
The deviations $\delta X_t^0, \delta \bar{z}_t, \delta z_t, \delta z_t^0$ have linear relationships with the initial errors $e^0$ and $\bar{e}$, and $\delta z_t^i$ has a linear relationship with the initial errors $e^i,e^0$ and $\bar{e}$. These linear relationships are listed as follows:
\begin{equation}
\left\{
\begin{split}
&\delta X_t^0=\Phi_t^{12}e^0+\Phi_t^{13}\bar{e},\delta z_t^0=\Phi_t^{22}e^0+\Phi_t^{23}\bar{e},\\
&\delta \bar{z}_t=\Phi_t^{32}e^0+\Phi_t^{33}\bar{e}, \delta z_t=\Phi_t^{42}e^0+\Phi_t^{43}\bar{e},\\
&\delta z^i_t=\Phi_t^x(e^i+\mathcal{M}_t^0e^0+\bar{\mathcal{M}}_t\bar{e}).
\end{split}
\right.
\end{equation}
\end{theorem}
\section{Estimate Initial Errors via Discrete Information}
In this section, we propose an initial error estimation method based on maximum likelihood estimation (MLE) \cite{casella2002,pawitan2001}. By applying the linear relationships in Section 3, we rewrite the model and transfer the error estimation problem into a parameter estimation problem. 

Consider error estimations at time $\tau, 0<\tau<T$. Set a time sequence $0=t_0<t_1<...<t_k=\tau$. Each minor agent $\mathcal{A}_i$ records the discrete time samples $\{X_{t_s}^i, X^0_{t_s}, z^i_{t_s}, 0\leq s\leq k\}$, and the major agent $\mathcal{A}_0$ records the discrete time samples $\{X_{t_s}^i, X^0_{t_s}, z^0_{t_s}, 0\leq s\leq k, i\in\mathcal{I}_{obs}\}$.

We consider the situation where $\mathcal{A}_i$ estimates $e^i,e^0,\bar{e}$ and $\mathcal{A}_0$ estimates $e^0$ and $\bar{e}$.
\subsection{Linear Parameterization of Initial Errors}
We substitute the error--deviation relationships derived in Section~III into the evolution equations of the observable states, so that all unknown initial errors enter the drift linearly.
\begin{equation}\label{eq:z_parameter}
\left\{
\begin{split}
&z_t^i-z_t=\Phi_t^xe^i+(\Phi_t^x\mathcal{M}_t^0-\Phi_t^{42})e^0+(\Phi_t^x\bar{\mathcal{M}}_t-\Phi_t^{43})\bar{e},\\
&z_t^0-z_t=(\Phi_t^{22}-\Phi_t^{42})e^0+(\Phi_t^{23}-\Phi_t^{43})\bar{e},\\
&z_t^0-z_t^i=(\Phi_t^{22}-\Phi_t^x\mathcal{M}_t^0)e^0+(\Phi_t^{23}-\Phi_t^x\bar{\mathcal{M}}_t)\bar{e}-\Phi_t^x e^i.
\end{split}
\right.
\end{equation}
A minor agent $\mathcal{A}_i$ has access to discrete time samples of $X^i, X^0, z^i$. Rewriting the dynamic equations \eqref{eq:Xi}, \eqref{eq:X0} and the evolution equation of $z^i$ and $z^0$, we have
\begin{equation}\label{eq:part_i}
d\mathcal{X}^i_t=(\mathbb{A}_t^1\mathcal{X}^i_t+\beta_t)dt-\mathbb{K}^1_t\mathcal{E}^idt+\mathbb{D}d\mathbb{W}_t^i,
\end{equation}
where
\begin{equation*}
\begin{split}
&\mathbb{A}_t^1:=\begin{pmatrix}
A-MP_t^x & G-MP_t^{31} & F-M(P_t^{32}-P_t^x)\\
0 & A_0-M_0P_t^{11} & F_0-M_0P_t^{12}\\
0 & G-MP_t^{31}&  A+F-MP_t^{32}
\end{pmatrix},\\
&\mathcal{X}_t^i:=\begin{pmatrix}
X_t^i\\
X_t^0\\
z_t^i
\end{pmatrix},
\beta_t:=\begin{pmatrix}
-M\mathcal{G}_t^3\\
-M_0\mathcal{G}_t^1\\
-M\mathcal{G}_t^3
\end{pmatrix},
\mathbb{D}:=\begin{pmatrix}
D & 0\\
0 & D_0\\
0 & 0
\end{pmatrix},
\end{split}\end{equation*}
\begin{equation*}\begin{split}
&\mathbb{W}_t^i:=\begin{pmatrix}
W_t^i\\
W_t^0
\end{pmatrix},
\mathcal{E}^i:=\begin{pmatrix}
e^i\\
e^0\\
\bar{e}
\end{pmatrix},\mathbb{K}_t^1:=\begin{pmatrix}
F\Phi_t^x & K_t & \bar{K}_t\\
K^{01}_t\Phi_t^x & K^0_t &\bar{K}^0_t\\
0 & 0 & 0
\end{pmatrix},\\
&K_t=F(\Phi_t^x\mathcal{M}_t^0-\Phi_t^{42}), \bar{K}_t=F(\Phi_t^x\bar{\mathcal{M}}_t-\Phi_t^{43}),K^{01}_t:=F_0\\
&-M_0P_t^{12},K^0_t=(F_0-M_0P_t^{12})\Phi_t^x\mathcal{M}_t^0+M_0P_t^{12}\Phi_t^{22}-F_0\\
&\Phi_t^{42},\bar{K}_t^0=(F_0-M_0P_t^{12})\Phi_t^x\bar{\mathcal{M}}_t+M_0P_t^{12}\Phi_t^{23}-F_0\Phi_t^{43}.
\end{split}
\end{equation*}

The major agent has access to discrete time samples of $X^0, X^j, z^0, j\in\mathcal{I}_{obs}$. Let $|\mathcal{I}_{obs}|=q$, $\mathcal{X}^{obs}_t:=((X_t^{j_1})^\top, ..., (X_t^{j_q})^\top)^\top$, $\mathcal{E}^{obs}:=((e^{j_1})^\top, ..., (e^{j_q})^\top)^\top$, $\mathcal{I}_{obs}:=\{j_s, 1\leq s\leq q\}$. Rewriting the dynamics, we have
\begin{equation}\label{eq:par_0}
d\mathcal{X}^0_t=(\mathbb{A}_t^2\mathcal{X}^0_t+\beta_t^0)dt-\mathbb{K}_t^0\mathcal{E}^0dt+\mathbb{D}^0d\mathbb{W}^0_t,
\end{equation}
where
\begin{equation*}
\begin{split}
&\mathcal{X}^0:=\begin{pmatrix}
X_t^0\\
\mathcal{X}_t^{obs}\\
z_t^0
\end{pmatrix},
\mathcal{E}^0:=\begin{pmatrix}
e^0\\
\mathcal{E}^{obs}\\
\bar{e}
\end{pmatrix},
\mathbb{W}^0_t:=\begin{pmatrix}
W_t^0\\
\mathcal{W}_t^{obs}\\
0
\end{pmatrix},\\
&\mathbb{A}_t^2:=\begin{pmatrix}
A_0-M_0P_t^{11} &0 &F_0-M_0P_t^{12}\\
g_t^{obs} & \mathbb{A}_t^{obs} &f_t^{obs}\\
G-MP_t^{31} &0 & A_t^{22}
\end{pmatrix}, \mathcal{W}_t^{obs}:=\\
&\begin{pmatrix}
W_t^{j1}\\
\vdots\\
W_t^{jq}
\end{pmatrix},\beta_t^0:=\begin{pmatrix}
-M_0\mathcal{G}_t^1\\
-I^q M\mathcal{G}_t^3\\
-M\mathcal{G}_t^3
\end{pmatrix},
\mathbb{D}^0:=\begin{pmatrix}
D_0 & 0\\
0 & \mathbb{D}^{obs}\\
0 & 0
\end{pmatrix}, \\
&\mathbb{D}^{obs}:=\mathrm{diag}(D,\ldots,D), \mathbb{A}^{obs}_t:=\mathrm{diag}(a_t^1,\ldots,a_t^1),\\
&\mathbb{K}_t^{obs}:=\mathrm{diag}(k_t^1,\ldots,k_t^1), I^q:=(I_d,\ldots,I_d)^\top_q,\\
&\mathbb{K}_t^0:=\begin{pmatrix}
F_0(\Phi_t^{22}-\Phi_t^{42})&0&F_0(\Phi_t^{23}-\Phi_t^{43})\\
I^q\mathcal{K}_t^1 &\mathbb{K}_t^{obs} &I^q\bar{\mathcal{K}}_t^1\\
0&0&0
\end{pmatrix},\\
&g_t^{obs}:=I^q(G-MP_t^{31}), f_t^{obs}:=I^q(F-M(P_t^{32}-P_t^x)),\\
&a_t^1=(A-MP_t^x), k_t^1=M(P_t^{32}-P_t^x)\Phi_t^x,\mathcal{K}_t^1:=(F-M\\
&(P_t^{32}-P_t^x))\Phi_t^{22}-F\Phi_t^{42}+M(P_t^{32}-P_t^x)\Phi_t^x\mathcal{M}_t^0,\bar{\mathcal{K}}_t^1:=\\
&(F-M(P_t^{32}-P_t^x))\Phi_t^{23}-F\Phi_t^{43}+M(P_t^{32}-P_t^x)\Phi_t^x\bar{\mathcal{M}}_t.
\end{split}
\end{equation*}
\subsection{Parameter Estimation Problem}
According to the last subsection, we can turn the error estimation problem to parameter estimation problems.

For a minor agent $\mathcal{A}_i$, it solves the following problem:
\begin{problem}[Parameter Estimation for $\mathcal{A}_i$]\label{prob:minor}For known observation values $\mathcal{X}^i_{t_s}, 0\leq s \leq k$, estimate $\mathcal{E}^i$. Where $\mathcal{X}^i$ satisfies \eqref{eq:part_i}. 
\end{problem}

For the major player $\mathcal{A}_0$, it solves the following problem:
\begin{problem}[Parameter Estimation for $\mathcal{A}_0$]\label{prob:major}For known observation values $\mathcal{X}^0_{t_s}, 0\leq s \leq k$, estimate $e^0,\bar{e}$. Where $\mathcal{X}^0$ satisfies \eqref{eq:par_0}. 
\end{problem}

Then we derive the form of the transition probability densities of $\mathcal{X}^i$ and $\mathcal{X}^0$ using the standard linear-Gaussian transition formula \cite{oksendal2003,kailath2000,pawitan2001}.
\subsection{Transition Probability Densities}
Solving \eqref{eq:part_i} and \eqref{eq:par_0}, for $0\leq s<t\leq T$, we have
\begin{equation}\label{eq:sol_obs}
\left\{
\begin{split}
&\mathcal{X}_t^i=\Xi_t\big(\Xi_s^{-1}\mathcal{X}_s^i+\int_s^t\Xi_u^{-1}\beta_udu-\int_s^t\Xi_u^{-1}\mathbb{K}_u^1du\mathcal{E}^i\\
&+\int_s^t\Xi_u^{-1}\mathbb{D}d\mathbb{W}_u^i\big),d\Xi_t=\mathbb{A}_t^1\Xi_tdt, \Xi_0=I_{2d+d_0},\\
&\mathcal{X}_t^0=\Psi_t\big(\Psi_s^{-1}\mathcal{X}_s^0+\int_s^t\Psi_u^{-1}\beta^0_udu-\int_s^t\Psi_u^{-1}\mathbb{K}_u^0du\mathcal{E}^0\\
&+\int_s^t\Psi_u^{-1}\mathbb{D}^0d\mathbb{W}_u^0\big),d\Psi_t=\mathbb{A}_t^2\Psi_tdt, \Psi_0=I_{(q+1)d+d_0}.
\end{split}
\right.
\end{equation}
Then we have the following theorem on transition probability densities
\begin{theorem}\label{thm:dens} If $\Sigma(t|s), \Sigma^0(t|s)\succ0$, for given $\mathcal{E}^i$, $\mathcal{E}^0$, $\mathcal{X}^i_t$ conditional on $\mathcal{X}_s^i=x$ is Gaussian, the transition probability density of $\mathcal{X}^i$ from $s$ to $t, t>s$ is
\begin{equation}\label{eq:dens_i}
\begin{split}
&\mathcal{T}(t,y|s,x;\mathcal{E}^i)=\frac{1}{\sqrt{(2\pi)^{2d+d_0}}|\Sigma|}\exp^{-\frac{1}{2}(y-\mu)^\top\Sigma^{-1}(y-\mu)},\\
&\mu(t|s,x;\mathcal{E}^i)=F(t|s,x)-\Xi_t\int_s^t\Xi_u^{-1}\mathbb{K}_u^1du\mathcal{E}^i,\\
&\Sigma(t|s)=\Xi_t\int_s^t\Xi_u^{-1}\mathbb{D}\mathbb{D}^\top\Xi_u^{-\top}du\Xi_t^\top,\\
&F(t|s,x)=\Xi_t\Xi^{-1}_sx+\Xi_t\int_s^t\Xi_u^{-1}\beta_udu,
\end{split}
\end{equation} 
and $\mathcal{X}_t^0$ conditional on $\mathcal{X}^0_s=x$ is Gaussian, the transition probability density of $\mathcal{X}^0$ from $s$ to $t, t>s$ is
\begin{equation}\label{eq:dens_0}
\begin{split}
&\mathcal{T}^0(t,y|s,x;\mathcal{E}^0)=\frac{\exp^{-\frac{1}{2}(y-\mu^0)^\top\Sigma_0^{-1}(y-\mu^0)}}{\sqrt{(2\pi)^{(q+1)d+d_0}}|\Sigma_0|},\\
&\mu^0(t|s,x;\mathcal{E}^0)=F^0(t|s,x)-\Psi_t\int_s^t\Psi_u^{-1}\mathbb{K}_u^0du\mathcal{E}^0,\\
&\Sigma^0(t|s)=\Psi_t\int_s^t\Psi_u^{-1}\mathbb{D}^0(\mathbb{D}^0)^\top\Psi_u^{-\top}du\Psi_t^\top,\\
&F^0(t|s,x)=\Psi_t\Psi^{-1}_sx+\Psi_t\int_s^t\Psi_u^{-1}\beta^0_udu,
\end{split}
\end{equation} 
\end{theorem}

\subsection{Likelihood Function}
According to Theorem \ref{thm:dens}, for given time sequence $0=t_0<t_1<...<t_k=\tau$, we can derive the likelihood function for major and minor agents. We define
\begin{equation}
\begin{split}
&F_s(x):=F(t_{s+1}|t_s,x), H_s:=\Xi_{t_{s+1}}\int_{t_s}^{t_{s+1}}\Xi_u^{-1}\mathbb{K}_u^1du,\\
&\Sigma_s:=\Sigma(t_{s+1}|t_s), H_s^0:=\Psi_{t_{s+1}}\int_{t_s}^{t_{s+1}}\Psi_u^{-1}\mathbb{K}_u^0du,\\
&F_s^0(x):=F^0(t_{s+1}|t_s,x), \Sigma^0_s:=\Sigma^0(t_{s+1}|t_s).
\end{split}
\end{equation}
For samples $\mathcal{X}_{t_s}^i,s=0,...,k$, the log-likelihood function is 
\begin{equation}\label{eq:like_i}
\ell_i(\mathcal{E}^i)=\sum_{s=0}^{k-1}\log\mathcal{T}(t_{s+1},\mathcal{X}^i_{t_{s+1}}|t_s,\mathcal{X}^i_{t_{s}};\mathcal{E}^i),
\end{equation}
and for $\mathcal{X}_{t_s}^0,s=0,...,k$, the log-likelihood function is
\begin{equation}\label{eq:like_0}
\ell_0(\mathcal{E}^0)=\sum_{s=0}^{k-1}\log\mathcal{T}^0(t_{s+1},\mathcal{X}^0_{t_{s+1}}|t_s,\mathcal{X}^0_{t_{s}};\mathcal{E}^0),
\end{equation}
Since the major agent is only interested in $(e^0,\bar{e})$, while $\mathcal{E}^{obs}$ enters the model only as a nuisance parameter, we define the profile log-likelihood function \cite{casella2002,pawitan2001}
\[\ell_p^0(e^0,\bar{e}):=\sup_{\mathcal{E}^{obs}}\ell_0(e^0,\mathcal{E}^{obs},\bar{e}).\]
\subsection{Distributed Error Estimation based on MLE}
Define $\mathcal{A}_i$'s maximum likelihood estimation for $\mathcal{E}^i$ as $\hat{\mathcal{E}}^i$, then a minor agent $\mathcal{A}_i$ solves $\hat{\mathcal{E}}^i$ from
\begin{equation}\label{eq:estimate_i}
\hat{\mathcal{E}}^i=\arg\max_{\mathcal{E}^i\in\mathbb{R}^{3d}}\ell_i(\mathcal{E}^i).
\end{equation}
Define $\mathcal{A}_0$'s maximum likelihood estimation for $(e^0,\bar{e})$ as $(e^0_0,\bar{e}_0)$, then the major agent $\mathcal{A}_0$ solves $\hat{\mathcal{E}}^0$ from
\begin{equation}\label{eq:estimate_0}
(e^0_0,\bar{e}_0)=\arg\max_{e^0,\bar{e}\in\mathbb{R}^{d}}\ell_p^0(e^0,\bar{e}).
\end{equation}
To make sure the estimators are well-defined, we introduce the following assumption.
\begin{assumption}[Identifiability assumption]
Assume $\Sigma_s\succ0$ and $\Sigma_s^0\succ0$ for $s=0,\ldots,k-1$, and
\begin{enumerate}
\item[(i)] For each minor agent $\mathcal{A}_i$,
\[
\mathcal I_i^{(k)}(\tau):=\sum_{s=0}^{k-1} H_s^\top \Sigma_s^{-1} H_s \succ 0 .
\]
\item[(ii)] For the major agent,
\[
\mathcal I_0^{(k)}(\tau):=\sum_{s=0}^{k-1} (H_s^0)^\top (\Sigma_s^0)^{-1} H_s^0 \in \mathbb R^{(q+2)d\times(q+2)d}.
\]
Define $\theta:=\bigl((e^0)^\top,(\bar e)^\top\bigr)^\top$
and $\eta:=\mathcal E^{obs}$, and
\[
S_\theta :=
\begin{pmatrix}
I_d & 0 & 0\\
0 & 0 & I_d
\end{pmatrix},
S_\eta^\top := \begin{pmatrix} 0 \\
 I_{qd} \\
  0\end{pmatrix},
\]
\begin{equation*}
\begin{split}
&\mathcal I_{\theta\theta}^{(k)}:=S_\theta \mathcal I_0^{(k)}(\tau) S_\theta^\top,
\mathcal I_{\eta\eta}^{(k)}:=S_\eta \mathcal I_0^{(k)}(\tau) S_\eta^\top,\\
&\mathcal I_{\theta\eta}^{(k)}:=S_\theta \mathcal I_0^{(k)}(\tau) S_\eta^\top,
\mathcal I_{\eta\theta}^{(k)}:=(\mathcal I_{\theta\eta}^{(k)})^\top.
\end{split}
\end{equation*}
Assume $\mathcal I_{\eta\eta}^{(k)}\succ0$ and the effective information matrix
\[
\mathcal I_{0,\mathrm{eff}}^{(k)}(\tau)
:=
\mathcal I_{\theta\theta}^{(k)}
-\mathcal I_{\theta\eta}^{(k)}\bigl(\mathcal I_{\eta\eta}^{(k)}\bigr)^{-1}\mathcal I_{\eta\theta}^{(k)}
\succ 0.
\]
\end{enumerate}
\end{assumption}
Under this assumption, the log-likelihood is strictly concave in the parameters, and hence the estimators are uniquely defined.

Since the form of log-likelihood function $\ell_i$ can be derived only according to system parameters and samples recorded by $\mathcal{A}_i$, a minor agent $\mathcal{A}_i$ can estimate the initial errors $\mathcal{E}^i$ by solving \eqref{eq:estimate_i}. 
Also, the form of $\ell_p^0$ can be derived only according to system parameters and samples recorded by $\mathcal{A}^0$, the major agent $\mathcal{A}_0$ can estimate the initial errors $e^0,\bar{e}$ by solving \eqref{eq:estimate_0}.
\section{Estimate-Based Mean-Field Reconstruction and Strategy Modification}
In this section, we use the estimated initial errors together with the linear error-propagation relations to reconstruct the actual current mean field and modify the agents’ strategies accordingly at time $\tau$. We then analyze the effect of estimation errors on the reconstructed mean field and the modified controls.  
\subsection{Estimation of the Current Mean Field}
Recall that 
\begin{equation}
\left\{
\begin{split}
&z_t=z_t^i-\Phi_t^xe^i-(\Phi_t^x\mathcal{M}_t^0-\Phi_t^{42})e^0-(\Phi_t^x\bar{\mathcal{M}}_t-\Phi_t^{43})\bar{e},\\
&z_t=z_t^0-(\Phi_t^{22}-\Phi_t^{42})e^0-(\Phi_t^{23}-\Phi_t^{43})\bar{e}.
\end{split}
\right.
\end{equation}

For a minor agent $\mathcal{A}_i$, it has access to current updated mean field state $z_\tau^i$. Hence, using the estimator $\hat{\mathcal{E}}^i$, $\mathcal{A}_i$ can estimate current mean field state by
\begin{equation}
\hat{z}^i_\tau=z_\tau^i-C_\tau^i\hat{\mathcal{E}}^i, C_t^i=(\Phi_t^x, \Phi_t^x\mathcal{M}_t^0-\Phi_t^{42}, \Phi_t^x\bar{\mathcal{M}}_t-\Phi_t^{43}).
\end{equation}

For the major agent $\mathcal{A}_0$, it has access to $z_\tau^0$. Hence, using the estimator $e_0^0,\bar{e_0}$, $\mathcal{A}_0$ can estimate current mean field state by
\begin{equation}
\hat{z}^0_\tau=z_\tau^0-(\Phi_\tau^{22}-\Phi_\tau^{42})e^0_0-(\Phi_\tau^{23}-\Phi_\tau^{43})\bar{e}_0.
\end{equation}
\subsection{Estimate-Based Strategy Modification}
We now propose a one-shot estimate-based modification at time $\tau$. Specifically, after computing the estimator for current mean field state, each agent resets its feedback control law. Since for $\mathcal{A}_j,j\geq0$, $\hat{z}^j_\tau$ is $\mathcal{A}_j$'s only estimation of $z_\tau$, it's natural to assume $\mathcal{A}_j$ takes it as the correct current mean field state, which leads to the following assumption:
\begin{assumption}\label{assume_common}
$\mathcal{A}_j$ treats $\hat{z}_\tau^j$ as its best available estimation of the current actual mean field state, and proceeds as if other agents have the same estimation when applying strategy modifications, $j\geq0$.
\end{assumption}

Based on the above assumption, in $\mathcal{A}_j$'s belief, the game beginning at $t=\tau$ becomes a game under correct information with initial mean field state $\hat{z}_\tau^j$. Then its update law for $z_t$ changes to
\begin{equation}\label{eq:close_znew}
\left\{
\begin{split}
&d\hat{z}^j_t=\big((A+F-BR^{-1}B^\top P_t^{32})\hat{z}^j_t+(G-BR^{-1}B^\top \\
&P_t^{31})X_t^0-BR^{-1}B^\top \mathcal{G}_t^3\big)dt, \hat{z}^j_\tau=\hat{z}^j_\tau, t\in[\tau,T].
\end{split}
\right.
\end{equation}
$\mathcal{A}_0$'s estimate-based reset law is
\begin{equation}
u_t^0=-R_0^{-1}B_0^\top(P_t^{11}X_t^0+P_t^{12}\hat{z}_t^0+\mathcal{G}_t^1),
\end{equation}
and a minor agent $\mathcal{A}_i$'s estimate-based reset law is
\[u_t^i=-R^{-1}B^\top(P_t^xX_t^i+(P_t^{32}-P_t^x)\hat{z}_t^i+P_t^{31}X_t^0+\mathcal{G}_t^3).\]

The above modification should be viewed as an estimate-based reset rule rather than a new equilibrium. 
\subsection{Estimation Error Effect}
We next quantify how the estimation errors propagate to the mean field reconstruction and to the reset controls. Define
\[\tilde{\mathcal{E}}^i:=\hat{\mathcal{E}^i}-\mathcal{E}^i, \tilde{e}_0^0=e_0^0-e^0, \tilde{e}_0:=\bar{e}_0-\bar{e},\]
and 
\[\tilde{z}_t^j:=\hat{z}_t^j-z_t,j\geq0.\]
Then we have
\[\tilde{z}_\tau^i=-C_{\tau}^i\tilde{\mathcal{E}}^i, i\geq 1,\]
and
\[\tilde{z}_\tau^0=-(\Phi_\tau^{22}-\Phi_{\tau}^{42})\tilde{e}_0^0-(\Phi_{\tau}^{23}-\Phi_\tau^{43})\tilde{e}_0.\]
Define $z^{new,c}_t, X_t^{0,c1}, t\in[\tau,T]$ as the mean field state and the state of major agent when agents get actual mean field at time $\tau$. Define $\bar{z}_t:=\lim_{N\rightarrow\infty}\frac{1}{N}\sum_{i=1}^N\hat{z}_t^i$, $\tilde{\mathcal{E}}:=\lim_{N\rightarrow\infty}\frac{1}{N}\sum_{i=1}^N\tilde{\mathcal{E}}^i$ and
\begin{equation*}
\begin{split}
&\delta \hat{z}_t^j=\hat{z}_t^j-z_t^{new,c}, \delta \hat{z}_t:=\bar{z}_t-z_t^{new,c},\\ &\delta z^{new}_t:=z_t-z_t^{new,c},\delta \hat{X}_t^0=X_t^0-X_t^{0,c1}.
\end{split}
\end{equation*}
Then for $\hat{\mathcal{X}}_t:=((\delta \hat{X}_t^0)^\top,(\delta \hat{z}_t^0)^\top,(\delta \hat{z}_t)^\top,(\delta z^{new}_t)^\top)^\top$, 
\begin{equation}
\begin{split}
&d\hat{\mathcal{X}}_t=\mathbb{A}_t^0\hat{\mathcal{X}}_tdt,\hat{\mathcal{X}}_\tau=(0^\top,(\delta\hat{z}_\tau^0)^\top,(\delta\hat{z}_\tau)^\top,0^\top)^\top.
\end{split}
\end{equation}
According to the discussion in Section 3, we have
\[\hat{\mathcal{X}}_t=\Phi_t\Phi_\tau^{-1}\hat{\mathcal{X}}_\tau,\]
which means, the deviations $\delta z^{new}_t, \delta \hat{X}_t^0$ have linear relationships with the estimation errors $\tilde{\mathcal{E}}, \tilde{e}_0^0$ and $\tilde{e}_0$.
\section{Properties of the Error Estimation}  
We now turn to the initial error estimators and study their statistical properties. In particular, we characterize the estimation error distributions for both minor and major agents, and then examine how the choice of the observation set affects the major agent’s estimation precision in the present symmetric setting.
\subsection{Estimation Error Distribution}\label{subsec:V-C}
In this subsection, we characterize the estimation error distribution. Define
\[
Y_s^{\,i}:=\mathcal{X}_{t_{s+1}}^{i}-F_s(\mathcal{X}_{t_s}^{i}), Y_s^{\,0}:=\mathcal{X}_{t_{s+1}}^{0}-F_s^0(\mathcal{X}_{t_s}^{0}).
\]
By Theorem~4.1, for each $s$,
\begin{equation}
Y_s^{i}=-H_s\mathcal{E}^i+\varepsilon_s^{i}, \varepsilon_s^{i}\sim\mathcal N(0,\Sigma_s), 
\label{eq:linreg-minor}
\end{equation}
\begin{equation}
Y_s^{0}=-H_s^0\mathcal{E}^0+\varepsilon_s^{0},\varepsilon_s^{0}\sim\mathcal N(0,\Sigma_s^0).
\label{eq:linreg-major}
\end{equation}
Define
\begin{equation}
\begin{split}
&Y^{i}:=\begin{pmatrix}Y_0^{i}\\ \vdots\\ Y_{k-1}^{i}\end{pmatrix},
\varepsilon^{i}:=\begin{pmatrix}\varepsilon_0^{\,i}\\ \vdots\\ \varepsilon_{k-1}^{\,i}\end{pmatrix},
\mathsf H:=\begin{pmatrix}H_0\\ \vdots\\ H_{k-1}\end{pmatrix},\\
&\mathsf\Sigma:=\mathrm{diag}(\Sigma_0,\ldots,\Sigma_{k-1}),
\end{split}
\end{equation}
and similarly $Y^{\,0},\varepsilon^{\,0},\mathsf H^0,\mathsf\Sigma^0$ for the major agent.
Then \eqref{eq:linreg-minor}--\eqref{eq:linreg-major} become
\begin{equation}\label{eq:Y}
\begin{split}
&Y^{i}=-\mathsf H\mathcal{E}^i+\varepsilon^{i},\varepsilon^{i}\sim \mathcal N(0,\mathsf\Sigma),\\
&Y^{0}=-\mathsf H^0\mathcal{E}^0+\varepsilon^{0},\varepsilon^{0}\sim \mathcal N(0,\mathsf\Sigma^0).
\end{split}
\end{equation}

\begin{theorem}[minor agents]\label{thm:finite-minor}
Under Assumption~4.1(i), the estimated error $\hat{\mathcal{E}}_i$ satisfies
\begin{equation}
\hat{\mathcal{E}}_i-\mathcal{E}_i \sim \mathcal N\!\left(0,\ \bigl(\mathcal I_i^{(k)}(\tau)\bigr)^{-1}\right).
\label{eq:dist-minor-mle}
\end{equation}
\end{theorem}

\begin{proof}
By (35), the log-likelihood of $\mathcal{E}_i$ is
\[
\ell_i(\mathcal{E}^i)
=
-\frac12\bigl(Y^{i}+\mathsf H \mathcal{E}^i\bigr)^\top \mathsf\Sigma^{-1}\bigl(Y^{i}+\mathsf H \mathcal{E}^i\bigr).
\]
Assumption~4.1(i) is equivalent to $\mathsf H^\top\mathsf\Sigma^{-1}\mathsf H=\mathcal I_i^{(k)}(\tau)\succ 0$,
hence $\ell_i$ is strictly concave and admits a unique maximizer.
Setting the gradient to zero yields
\[
\mathsf H^\top\mathsf\Sigma^{-1}\mathsf H\,\hat{\mathcal{E}}^i = -\mathsf H^\top\mathsf\Sigma^{-1}Y^{i}.
\]
Then we have
\[
\hat{\mathcal{E}}^i-\mathcal{E}^i
=
\bigl(\mathsf H^\top\mathsf\Sigma^{-1}\mathsf H\bigr)^{-1}\mathsf H^\top\mathsf\Sigma^{-1}\varepsilon^{i},
\]
the right-hand side is Gaussian with mean $0$ and covariance
\[
\bigl(\mathsf H^\top\mathsf\Sigma^{-1}\mathsf H\bigr)^{-1}\mathsf H^\top\mathsf\Sigma^{-1}\mathsf\Sigma\,\mathsf\Sigma^{-1}\mathsf H
\bigl(\mathsf H^\top\mathsf\Sigma^{-1}\mathsf H\bigr)^{-1}
=
\bigl(\mathcal I_i^{(k)}(\tau)\bigr)^{-1}.
\]
This proves \eqref{eq:dist-minor-mle}.
\end{proof}

\begin{theorem}[major agent]\label{thm:finite-major}
Under Assumption~4.1(ii), the estimator $\hat\theta$ satisfies
\[
\hat\theta-\theta \sim \mathcal N\!\left(0,\ \bigl(\mathcal I_{0,\mathrm{eff}}^{(k)}(\tau)\bigr)^{-1}\right).
\]
\end{theorem}
\begin{proof}
By \eqref{eq:like_0}, the full log-likelihood for $\mathcal{E}^0$ is quadratic and strictly concave under Assumption~4.1(ii),
hence the full MLE $\hat{\mathcal E}^0$ is unique and satisfies
\[
\hat{\mathcal E}^0-\mathcal E^0
=
\bigl((\mathsf H^0)^\top(\mathsf\Sigma^0)^{-1}\mathsf H^0\bigr)^{-1}(\mathsf H^0)^\top(\mathsf\Sigma^0)^{-1}\varepsilon^{0}.
\]
Therefore $\hat{\mathcal E}^0-\mathcal E^0$ is Gaussian with covariance
$(\mathcal I_0^{(k)}(\tau))^{-1}$.

Applying $\theta=S_\theta\mathcal E^0$ and $\eta=S_\eta\mathcal E^0$, the covariance of $\hat{\theta}-\theta$
is $(\mathcal I_{0,\mathrm{eff}}^{(k)}(\tau))^{-1}$, and hence
$\hat{\theta}-\theta\sim\mathcal N\bigl(0,(\mathcal I_{0,\mathrm{eff}}^{(k)}(\tau))^{-1}\bigr)$.

Since maximizing $\eta$ in the quadratic log-likelihood yields the same estimator as the $\theta$-component of the full MLE, the profile MLE $\hat\theta$ has the same distribution.
\end{proof}
\subsection{Index-Invariance at Fixed Observation Budget}\label{subsec:V-D}
The following proposition shows that, the estimation accuracy of $\theta$ is independent from specific choice of $\mathcal I_{\mathrm{obs}}$ with a fixed size $q$.
\begin{proposition}\label{prop:index-invariance-fixedq}
Let $\mathcal I_{\mathrm{obs}}=\{j_1,\ldots,j_q\}$ and $\mathcal I_{\mathrm{obs}}'=\{j_1',\ldots,j_q'\}$ be any two observation sets
with the same cardinality $q$. We have
\[
\mathcal I_{0,\mathrm{eff}}^{(k)}(\tau;\mathcal I_{\mathrm{obs}})
=
\mathcal I_{0,\mathrm{eff}}^{(k)}(\tau;\mathcal I_{\mathrm{obs}}').
\]
\end{proposition}

\begin{proofsketch}
In \eqref{eq:par_0}, the observed minors enter only through
$\mathcal X_t^{obs}$, $\mathcal E^{obs}$ and
$\mathcal{W}_t^{obs}$,
while all coefficient matrices (e.g., $\mathbb{A}_t^2,\mathbb{D}^0,\mathbb{K}_t^0$) are defined by block repetition and depend
only on $q$.

Therefore, replacing $\mathcal I_{\mathrm{obs}}$ by another set $\mathcal I_{\mathrm{obs}}'$ of the same size $q$ does not change the Fisher information for $\mathcal{E}^0$, and $\mathcal I_{0,\mathrm{eff}}^{(k)}(\tau;\mathcal I_{\mathrm{obs}})=\mathcal I_{0,\mathrm{eff}}^{(k)}(\tau;\mathcal I_{\mathrm{obs}}')$.
\end{proofsketch}
\section{Simulation}
We next illustrate the proposed estimation-modification framework through a representative simulation example. 
We set $T=2$, $\tau=1$, $q=10$, $\mathcal I_{\mathrm{obs}}=\{1,2,\dots,10\}$,
and $k=50,t_s=0.02s$.

The system parameters are chosen as
\begin{equation*}
\begin{split}
&A=0.3, B=0.7, G=-0.8, F=0.2,F_0=0.6,A_0=0.7,\\
&B_0=0.8,D=0.05,D_0=0.03, Q=0.1,Q_0=0.3,R=1,\\
&R_0=1,\gamma=0.5,\gamma_0=1.2,\eta=0.1,\eta_0=0.1,\lambda=1,\bar Q=\\
&0.2,\bar Q_0=0.1,\bar\gamma=0.5,\bar\gamma_0=1.2,\bar\eta=0.1, \bar\eta_0=0.1,\bar\lambda=0.3.
\end{split}
\end{equation*}

We select $100$ minor agents from an infinite number of agents. Their initial states are sampled from $\mathcal N(0,0.01)$. 
The private initial errors are sampled from $\mathcal N(6,0.5)$, the mean error is $\bar{e}=6$ and the major error is $e^0=2$. 
\begin{figure}[t]
    \centering
    \begin{subfigure}[t]{0.48\linewidth}
        \centering
        \includegraphics[width=\linewidth]{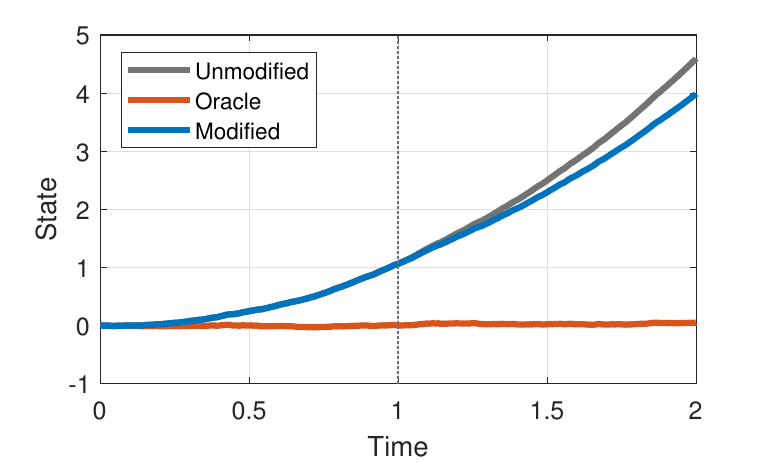}
    \end{subfigure}
    \hfill
    \begin{subfigure}[t]{0.48\linewidth}
        \centering
        \includegraphics[width=\linewidth]{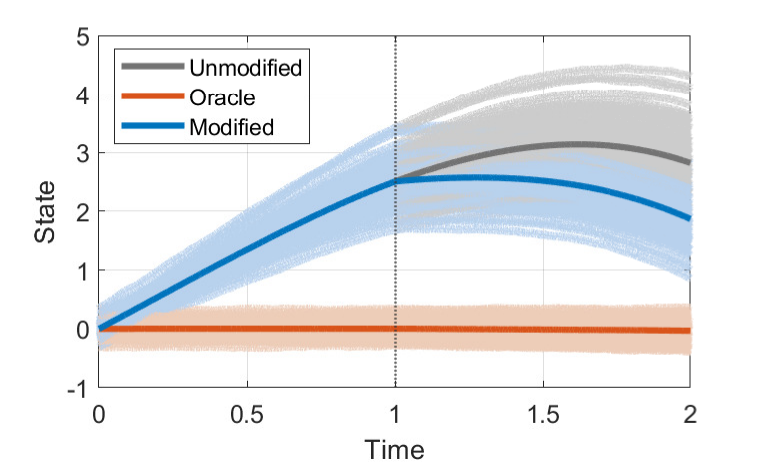}
    \end{subfigure}
    \caption{Trajectory comparison under oracle, unmodified, and modified strategies. 
    (a) The major-state trajectory; the vertical dotted line indicates the modification time.  
    (b) The minor-state trajectories together with the mean-field evolution.}
    \label{fig:traj_correction}
\end{figure}
Figure~\ref{fig:traj_correction} reports the trajectory-level effect of the proposed modification mechanism. 
After $t=\tau$, the modified trajectory moves visibly toward the oracle benchmark, whereas the unmodified trajectory continues to deviate because of the mismatch in the initial information. 
This phenomenon is observed both at the major-agent level and at the population level, indicating that the proposed strategy modification effectively mitigates the propagation of initial errors.

\begin{figure}[t]
    \centering
    \begin{subfigure}[t]{0.5\linewidth}
        \centering
        \includegraphics[width=\linewidth]{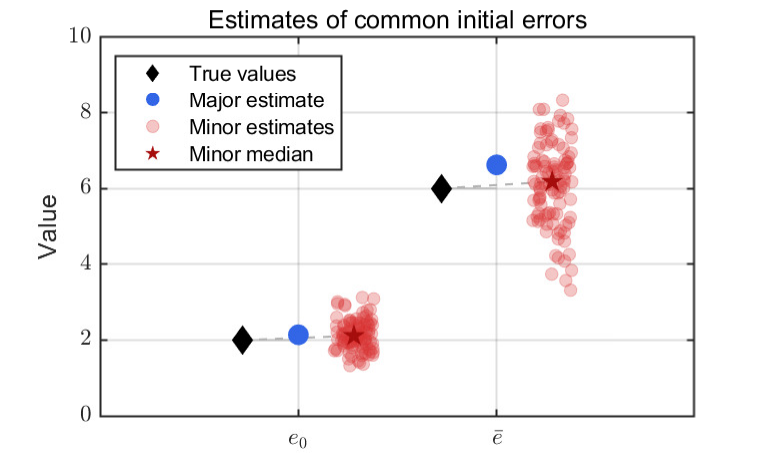}
    \end{subfigure}
    \hfill
    \begin{subfigure}[t]{0.48\linewidth}
        \centering
        \includegraphics[width=\linewidth]{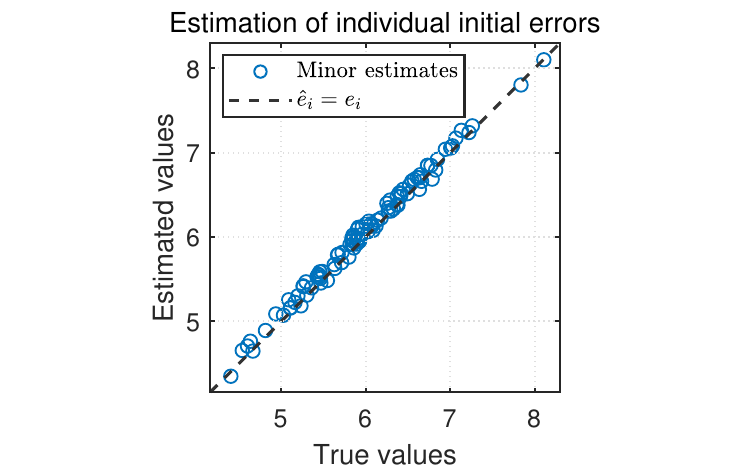}
    \end{subfigure}
    \caption{Estimation results for the initial errors. 
    (a) The estimates of the common initial errors $e_0$ and $\bar e$, including the ground truth, the major-agent estimate, and the minor-agent estimates. 
    (b) The true versus estimated private initial errors $e_i$; the diagonal line corresponds to perfect estimation.}
    \label{fig:error_estimation}
\end{figure}
To explain this improvement, Figure~\ref{fig:error_estimation} presents the associated estimation results. 
The left panel shows the estimates of the common initial errors, including $e_0$ and $\bar e$. 
The estimate obtained by the major agent is close to the ground truth, and the minor-agent estimates are concentrated around the true values. 
The right panel reports the estimation accuracy for the individual initial errors through a scatter plot of the true values versus the estimated values of $e_i$. 
The points lie close to the diagonal line, showing that the proposed estimator accurately recovers the individual error realizations. 
\section{Conclusion}
This paper studies major-minor LQ mean field games under erroneous initial information, focusing on the induced error propagation, discrete-time error estimation, and an estimate-based intermediate-time reset mechanism. Future work includes analyzing the limiting behavior of the estimation precision under increasingly frequent observations, and extending the framework to repeated modifications and more general nonlinear models.

\end{document}